\documentclass[a4paper,11pt]{article}
\usepackage{amsmath,amssymb}
\newtheorem{theorem}{Theorem}[section] % 1st argument is your name for it
\newtheorem{lemma}[theorem]{Lemma}     % 2nd argument is what is printed
\newtheorem{corollary}[theorem]{Corollary}
\newtheorem{proposition}[theorem]{Proposition}
\newenvironment{proof}{\normalsize {\sc Proof}:}{{\hfill $\Box$}}
\newenvironment{proofof}[1]{\normalsize {\sc Proof of #1}:}{{\hfill $\Box$}}

\def \N {\mathbb{N}}
\def \Z {\mathbb{Z}}

\def \p#1{{\rm pre}[#1]}

\def \f#1{{\rm f}[#1]}
\def \l#1{{\rm l}[#1]}

\newcommand{\LD}{{\rm LD}}
\newcommand{\RD}{{\rm RD}}
\newenvironment{mylist}{\begin{list}{}{
\setlength{\parskip}{0mm}
\setlength{\topsep}{1mm}
\setlength{\parsep}{0mm}
\setlength{\itemsep}{0.5mm}
\setlength{\labelwidth}{7mm}
\setlength{\labelsep}{3mm}
\setlength{\itemindent}{0mm}
\setlength{\leftmargin}{12mm}
\setlength{\listparindent}{6mm}
}}{\end{list}}
\title{Conjugacy in Artin groups of extra-large type} 
\author{Derek F. Holt and Sarah Rees}
\date{16th September 2013}
\begin{document}
\maketitle
\begin{abstract}
We describe a constructive, cubic time solution to the conjugacy problem in
Artin groups of extra-large type, which was proved solvable in those groups
in \cite{AppelSchupp83}. We use results from \cite{HoltRees,CHR} that
characterise geodesic words in those groups, as well as the description of
conjugacy between elements involving three or more generators that is given in
\cite{AppelSchupp83}.
\end{abstract}

\section{Introduction}
In \cite{AppelSchupp83}, Appel and Schupp used arguments from small cancellation
theory to prove that Artin groups of extra-large type have solvable conjugacy
problem, and this result was extended to Artin groups of large type in
\cite{Appel84}. In neither case was any analysis of the complexity of the
solution attempted.
In this paper, we address the complexity question in the extra-large case,
and prove the following result.

\begin{theorem}
\label{thm:main} 
The conjugacy of two elements represented as words of length at most $\ell$ in an $n$-generator Artin
group of extra-large type can be decided constructively in time $O(\ell)$ when
$n=2$ and $O(n\ell^3)$ when $n>2$.
\end{theorem}

An Artin group is defined by a presentation
\[ \langle x_1,\ldots ,x_n \mid {}_{m_{ij}}(x_i,x_j)=
{}_{m_{ij}}(x_j,x_i)\quad\hbox{\rm for each}\quad i \neq j \rangle, \]
where $(m_{ij})$ is a {\em Coxeter matrix}
(a symmetric $n \times n$ matrix with entries in
$\N \cup \{\infty\}$, $m_{ii}=1, m_{ij} \geq 2$, $\forall
i \neq j$),
and where
for generators $a,a'$ and $m \in \N$ we define ${}_m(a,a')$ to be the
word that is the product of $m$ alternating $a$s and $a'$s that
starts with $a$.
When $m_{ij}$ is infinite, there is no relation between $x_i$ and $x_j$.
The set $X=\{x_1,\ldots.x_n\}$ is called the {\em standard generating set};
an element of $X \cup X^{-1}$ is called a {\em letter}.
It is common to represent the Coxeter matrix graphically
using a {\em Coxeter graph} $\Gamma$ with $n$ vertices, $1,\ldots,n$ and
an edge labelled $m_{ij}$ joining $i,j$. We use the notation $G(\Gamma)$
to denote the Artin group defined by the Coxeter graph $\Gamma$.

An Artin group group has {\em large type} if all the integers $m_{ij}$ are
at least 3, and {\em extra-large type} if they are all at least 4.
Peifer proved in \cite{Peifer} that Artin groups of extra-large type are
biautomatic, and the authors proved in \cite{HoltRees} that
Artin groups of large type are automatic, which implies that their word
problem is solvable in quadratic time. Our proof of Theorem \ref{thm:main}
will use Peifer's result, and also make use of many of the techniques that we
developed in \cite{HoltRees}.

Let $G$ be an Artin group of extra-large type, with generating set $X$.
To prove Theorem~\ref{thm:main}
we describe a cubic time algorithm that, given words $u,v$ representing elements
$g,h \in G$ as input, determines whether or not $g,h$ are conjugate,
and if so produces a specific conjugating element. 
The algorithm varies depending on whether $u,v$ involve one, two, or three
generators, and the remaining sections of this article are organised
accordingly. 
In many situations the conjugacy problem reduces to the conjugacy problem within
a dihedral Artin group, and so we study that first, in Section~\ref{sec:dihedral}.
Then Section~\ref{sec:genpowers} examines the structure of conjugates of powers
of a generator. Section~\ref{sec:twogen} proves that when one of the input
words is written over two generators and is not a power of a generator, conjugacy can be
determined using the algorithm for the appropriate dihedral Artin groups.
Section~\ref{sec:complexity} analyses the complexity of the complete algorithm.

We use a number of results from \cite{HoltRees,CHR} describing the structure
of geodesic words in Artin groups of dihedral or large type, and the
reduction of words to that form;
a word $w$ is geodesic if its length $|w|$ is minimal 
over all words over $X$ that represent the same element of $G$ as $w$. 
Section~\ref{sec:artin} summarises what we need from \cite{HoltRees,CHR}.

In our algorithm we shall need both words and elements to be cyclically reduced.
As is standard, we call a word $w$ over $X$  {\em cyclically reduced} if it is
freely reduced and does not have the form $aua^{-1}$, for any letter $a$.
We call an element $g \in G$ cyclically reduced if all geodesic words
representing $g$ are cyclically reduced or, equivalently, if
the geodesic length $|aga^{-1}|$ of $aga^{-1}$ is no shorter than that of $g$, for any letter $a$.
(This equivalence follows from the fact that all relators of $G$ have
even length, and so $|aga^{-1}| < |g|$ implies $|aga^{-1}| = |g|-2$.)

The analysis of the complexity of our algorithm makes use of the fact,
proved in \cite{Peifer}, that Artin groups of extra-large type
are biautomatic with geodesic normal form.
In any such group, words over the generators can be reduced to normal form in
quadratic time and, if $w$ is already in normal form and $a$ is a letter,
then $aw$ and $wa$ (and hence also $awa^{-1}$) can
be reduced to normal form in linear time.

We also use \cite[Theorem 4$''$]{AppelSchupp83}:

\begin{theorem} [Artin, Schupp, 1983]
\label{thm:AppelSchupp}
Let $G$ be an Artin group of extra-large type. Then $G$ has solvable conjugacy
problem. Furthermore, if $u$ and $v$ are specially cyclically reduced words
involving at least three generators then u and v are conjugate in G
if and only if there are cyclic
permutations $u^*$ and $v^*$ of $u$ and $v$, respectively,
a generator $x_i$ occurring in both $u$ and $v$, and an exponent $n$
such that $u^* = x_i^n v^* x_i^{-n}$ in $G$.
\end{theorem}

A word is defined in \cite{AppelSchupp83} to be {\em specially cyclically
reduced} if none of its cyclic conjugates can be reduced in length in $G$ using
a certain type of substitution. So, in particular, if all cyclic conjugates of
the word are geodesics, then it is specially cyclically reduced.

\begin{proofof}{Theorem~\ref{thm:main}}
The case $n=2$ is proved as Proposition~\ref{dihedralconj}.
So from now on we assume that $n>2$.

The theorem is proved by the description of the algorithm, and analysis of
its correctness and complexity.
Here we describe the basic steps of the algorithm. The details
are given, with justification, in the sections that follow this one.
Complexity is analysed in Proposition~\ref{prop:complexity}.

Suppose that we are given words $u,v$ representing the elements $g,h$
for which we want to test conjugacy.

First we replace $u,v$
by normal form representatives, and check whether the geodesic lengths of the
elements represented by $aua^{-1}$ and $ava^{-1}$,
for $a \in X \cup X^{-1}$, are 
less than $|u|$ and $|v|$; if they are we replace $g,h$ by appropriate
conjugates and repeat this process as necessary.
Now $g,h$ are ensured cyclically reduced, and $u,v$ are in normal form.

Now we check to see if either $u$ or $v$ is a power of a single generator.
If so, then $g$ and $h$ can only be conjugate if $u$ and $v$ have the form
$x^k$ and $y^k$ for some integer $k$ and generators $x,y$ linked by a path
in the Coxeter graph on which all edges have odd labels. This is
justified in Proposition~\ref{genconj}.

Suppose next that one of $u,v$ involves just two generators. In that
case $g$ and $h$ can only be conjugate if $u$ and $v$ involve the same
two generators $x,y$, and if $g,h$ are conjugate in the subgroup of $G$
generated by $x,y$; this is justified by Proposition~\ref{2genconj}.
So in this case we can use the algorithm for dihedral Artin groups, and apply
Proposition~\ref{dihedralconj}.

Finally suppose that both $u$ and $v$ involve at least three generators. We
need to reduce to the situation where the conditions of
Theorem~\ref{thm:AppelSchupp} are satisfied.
So we check that all cyclic conjugates of the words $u,v$ are geodesic,
by reducing each to normal form.
If not, we replace $g,h$ by appropriate conjugates, and repeat as necessary.
At the end of this process, we have either landed in the one or two
generator case (and can use the methods referred to above) or
we can assume that $u,v$ satisfy the hypotheses of 
Theorem~\ref{thm:AppelSchupp}. That theorem together with 
Proposition \ref{prop:boundlen} ensures that a complete
test for conjugacy is provided by computing and comparing the normal
forms of all elements represented by $v^*$ and $a^nu^*a^{-n}$,
where $u^*$ is a cyclic conjugate of $u$, $v^*$ is a cyclic conjugate of $v$,
$a \in X \cup X^{-1}$, and $n \leq |u|$.
\end{proofof} 

In a forthcoming paper \cite{CHHR} we shall use some of the results proved
in this paper to show that, in an Artin group of extra-large type, the
set of conjugacy geodesics is regular. We define a word to be a {\em conjugacy
geodesic} if its length is minimal among representatives of elements in
its conjugacy class.

\section{Background on Artin groups}
\label{sec:artin}

In this section we give some background on Artin groups,
reproducing results already proved in \cite{HoltRees,CHR} that we shall
need in this article. Almost all of those results are valid for all
Artin groups of large type.

Suppose first that $G$ is any $n$-generator Artin group, given
in its standard presentation, $X$ its standard generating set.
A non-empty word over $X$ is called positive if it involves only positive powers
of generators, and negative if it involves only negative powers; otherwise
it is called unsigned.
We write $\f{w}$, $\l{w}$ to denote the first and last letters of a word $w$.
Where $a$ is a letter, we call the corresponding generator
($a$ or $a^{-1}$) the {\em name} of $a$.
For letters $a,b$ we extend the notation already used in the
Artin group presentations and define ${}_k(a,b)$ to be the product
of $k$ alternating $a$s and $b$s that starts with $a$,
and $(a,b)_k$ to be the product
of $k$ alternating $a$s and $b$s that ends with $b$.

For elements $g_1,\ldots,g_k \in G$, we say that the product $g_1\cdots g_k$
is geodesic if the element $g$ it represents has length equal to the sum
of the lengths $|g_i|$; equivalently $g_1\cdots g_k$ is a {\em geodesic
factorisation} of $g$. In this situation, $g_1$ is called a {\em left divisor} of $g$,
and $g_k$ a {\em right divisor} of $g$.

Now, for an integer $m$, let $G(m)$ denote the dihedral Artin group
\[ \langle  x_1,x_2 \mid {}_m(x_1,x_2) = {}_m(x_2,x_1) \rangle. \]
For $i,j \in \{1,\ldots,n\}$ with $i\neq j$, $G(i,j)$ will denote the
subgroup of 
$G$ generated by $x_i,x_j$; by \cite{Lek} $G(i,j)$ is isomorphic to the
dihedral Artin group $G(m_{ij})$.

Geodesics in dihedral Artin groups are recognised using a criterion
described in \cite{MM}.
Let $w$ be a freely reduced word over the generating set $\{x_1,x_2\}$
of the dihedral Artin group $G(m)$.
Then we define $p(w)$ to be the minimum of $m$ and the length of
the longest subword of $w$ of alternating $x_1$'s and $x_2$'s (that is, the
length of the longest subword of $w$ of the form $_r(x_1,x_2)$ or $_r(x_2,x_1)$).
Similarly, we define $n(w)$ to be the minimum of $m$ and the length of
the longest subword of $w$ of alternating $x_1^{-1}$'s and $x_2^{-1}$'s.
The following is proved in \cite[Proposition 4.3]{MM}:
\begin{proposition}
\label{dihedral_geodesics}
The word $w$ is geodesic in $G(m)$ if and only if
$p(w) +n(w) \leq m$.
If $p(w)+n(w) < m$, then $w$ is the unique geodesic representative of the group
element it defines, but if $p(w)+n(w)=m$ then there are other representatives.
\end{proposition}

A freely reduced, unsigned word $w$ over $\{x_1,x_2\}$ with
$p(w)+n(w)=m$ is defined to
be {\em critical} if it is has either of the forms
\[ {}_p(x,y)\xi (z^{-1},t^{-1})_n \quad{\rm or}\quad
 {}_n(x^{-1},y^{-1})\xi (z,t)_p. \]
with $p=p(w)$, $n=n(w)$ and $\{x,y\} = \{z,t\}=\{x_1,x_2\}$.
(Obviously these conditions put some restrictions on the subword $\xi $.)
We define a positive geodesic word $w$ to be critical if it has either
of the forms ${}_m(x,y) \xi$ or $\xi (x,y)_m$, and only the one positive alternating
subword of length $m$.
Similarly we define a negative geodesic word $w$ to be critical it is has either
of the forms ${}_m(x^{-1},y^{-1}) \xi$ or $\xi (x^{-1},y^{-1})_m$, and only
the one negative alternating
subword of length $m$.

The element $\Delta$ represented by the geodesic words $(x_1,x_2)_m$ and
$(x_2,x_1)_m$ conjugates $x_1$ to $x_2$, and $x_2$ to $x_1$ when $m$ is odd,
and is central when $m$ is even. For a word $\xi$, we denote by $\delta(\xi)$
the image of $\xi$ under the map which maps each generator to its
conjugate under $\Delta$. 

Now we define an involution $\tau$ on the set of critical words by
\begin{eqnarray*}
\tau({}_p(x,y)\,\xi\,(z^{-1},t^{-1})_n)
&:=& {}_n(y^{-1},x^{-1})\,\delta(\xi )\,(t,z)_p,\quad\quad(p,n>0)\\
\tau({}_n(x^{-1},y^{-1})\,\xi\,(z,t)_p) 
&:=& {}_p(y,x)\,\delta(\xi )\,(t^{-1},z^{-1})_n,\quad\quad(p,n>0)\\
\tau({}_m(x,y))&:=& {}_m(y,x),\\
\tau({}_m(x^{-1},y^{-1}))&:=& {}_m(y^{-1},x^{-1}),\\
\tau({}_m(x,y)\,\xi) &:=& \delta(\xi)\,(z,t)_m,
\quad\hbox{\rm where}\quad z=\l{\xi},\,\{x,y\}=\{z,t\},\\ 
\tau(\xi\,(x,y)_m) &:=& {}_m(t,z)\,\delta(\xi),
\quad\hbox{\rm where}\quad z=\f{\xi},\,\{x,y\}=\{z,t\},\\ 
\tau({}_m(x^{-1},y^{-1})\,\xi) &:=& \delta(\xi)\,(z^{-1},t^{-1})_m,
\quad\hbox{\rm where}\quad z=\l{\xi}^{-1},\,\{x,y\}=\{z,t\},\\ 
\tau(\xi\,(x^{-1},y^{-1})_m) &:=& {}_m(t^{-1},z^{-1})\,\delta(\xi),
\quad\hbox{\rm where}\quad z=\f{\xi}^{-1},\,\{x,y\}=\{z,t\}.
\end{eqnarray*}
In all cases, a critical word and its image under $\tau$ represent the same group
element; this is easily verified using equations such as
${}_p(x,y)={}_n(y^{-1},x^{-1})\Delta$.

A freely reduced non-geodesic word in a dihedral Artin group 
must contain a critical subword whose replacement by its image under
$\tau$ will produce a word admitting a free reduction \cite[Lemma 2.3]{HoltRees}.

We define a word $w$ to be over-critical if $p(w)+n(w) > m$. We can extend our
definition of $\tau$ to over-critical words via
\begin{eqnarray*}
\tau({}_p(x,y)\,\xi\,(z^{-1},t^{-1})_n)
&:=& {}_{m-p}(y^{-1},x^{-1})\,\delta(\xi )\,(t,z)_{m-n},\\
\tau({}_n(x^{-1},y^{-1})\,\xi\,(z,t)_p)
&:=& {}_{m-n}(y,x)\,\delta(\xi )\,(t^{-1},z^{-1})_{m-p}.
\end{eqnarray*}
An over-critical word and its image under $\tau$ represent the same
group element, but the image under $\tau$ is shorter.

Now, and for the remainder of this section, suppose that $G=G(\Gamma)$
is an Artin group of large type with generating set $X$.

We can extend the definition of critical and over-critical words to words
defined over any two of the generators of $X$
in the obvious way. Again, critical words
give us a criterion for recognising geodesics; a non-geodesic
word must contain a critical subword, and admit a sequence of $\tau$-moves
to overlapping subwords the last of which provokes a free reduction
\cite{HoltRees}. We call such a sequence a {\em reducing sequence}. The sequence
is {\em rightward} if each $\tau$-move is applied to the right of its
predecessor. Such a sequence of reductions to a word $w$ corresponds to a
{\em critical factorisation} of $w$ as $\alpha u_1\cdots u_k\beta$,
where each of $u_1,\l{\tau(u_1)}u_2,\ldots,\l{\tau(u_{k-1}}u_k$ is
critical, and where the last letter of $\tau(u_k)$ freely
cancels with the first letter of $\beta$.

We need some results from \cite{HoltRees,CHR} about geodesics words in $G$.
Proofs are supplied in \cite{HoltRees,CHR}.

\begin{proposition}\cite[Proposition 4.5]{HoltRees}
\label{twoposs}
Suppose that $v,w$ are any two geodesic words representing the same group
element, and that $\l{v} \ne \l{w}$.
Then
\begin{mylist}
\item[(1)] $\l{v}$ and $\l{w}$ have different names;
\item[(2)] The maximal 2-generator suffixes of $v$ and $w$ involve generators
with names equal to those of $\l{v}$ and $\l{w}$;
\item[(3)]
Any geodesic word equal in $G$ to $v$ must end in $\l{v}$ or in $\l{w}$.
\end{mylist}
Corresponding results apply if $\f{v} \ne \f{w}$.
\end{proposition}

\begin{corollary}\cite[Corollary 7.2]{CHR}
\label{twoposscor}
If $wa$ is a geodesic word for some $a \in A$, then $wa^k$ is
a geodesic word for all $k>1$.
\end{corollary}

\begin{proposition}\cite[Proposition 7.5]{CHR}
\label{ngreduceprop}
\begin{mylist}
\item[(1)]
Let $v,w$ be two geodesic words representing the same group
element $g$, with $\l{v} \ne \l{w}$. Then a single rightward critical sequence
can be applied to $v$ to yield a word ending in $\l{w}$.
\item[(2)] Let $v$ be a freely reduced non-geodesic word with $v=wa$ with
$a \in A$ and $w$ geodesic. Then $v$ admits a rightward length reducing sequence.
\end{mylist}
\end{proposition}

\begin{proposition}\cite[Proposition 7.3]{CHR}
\label{longestld}
Let $g \in G$ and $x_i,x_j \in X$ with $i \ne j$. Then $g$ has a unique left
divisor $\LD_{ij}(g) \in G(i,j)$ of maximal length. Furthermore, if $w$ is
any geodesic word representing $g$, and $u$ is the maximal
$\{x_i,x_j\}$-prefix of $w$,
then $\LD_{ij}(g) =_G u a^r$ for some $r \ge 0$ with
$a \in \{x_i^{\pm 1}, x_j^{\pm 1} \}$ and $|\LD_{ij}(g)| = |u| + r$.

Similarly, $g$ has a unique right divisor $\RD_{ij}(g) \in G(i,j)$ of
maximal length, to which the corresponding results apply.
\end{proposition}

The following two lemmas are stated and proved in \cite{CHR}
in a specific setting defined in that article. But it is clear that the
proofs of \cite{CHR} are valid more generally, and so can be applied in this
article. So we do not reprove the results as stated here.

\begin{lemma}\cite[Lemma 7.11]{CHR}
\label{geodfaclem1}
Suppose that $g=g_1g_2g_3$, with $g_2 \in G(i,j)$ not a power of a generator,
 and
$\RD_{ij}(g_1)=\LD_{ij}(g_3)=1$. Suppose also that $g_1g_2g_3$ is not a
geodesic factorisation of $g$.
Then $g_2=a^sb^t$
for some $s,t>0$ where $a,b\in  \{ x_i^{\pm 1},x_j^{\pm 1} \}$.
\end{lemma}

\begin{lemma} \cite[Lemma 7.12]{CHR}
\label{geodfaclem2}
Let $G=G(\Gamma)$ be an Artin group of extra-large type.
Suppose that $g=g_1g_2g_3$, with $g_2=a^sb^t$ with $a=x_i^{\pm 1},b=x_j^{\pm 1}$
$i\neq j,\,s,t>0$ and
$\RD_{ij}(g_1)=\LD_{ij}(g_3)=1$, and that $g_1g_2g_3$ is not a geodesic
factorisation of $g$.

Then there exists a letter $c=x_{i'}^{\pm 1}$ for $i' \neq i,j$,
$q>0$, and elements $e_1,e_2$ of $G$, such that  $g_1a^s =_G e_1 c^q$ and
$b^t g_3 =_G c^{-q} e_2$,
where $e_1c^q$, $c^{-q}e_2$ and  $e_1 e_2$ are all geodesic factorisations.
%NEW SENTENCE FOLLOWS
%Furthermore, the vertices corresponding to $x_i,x_j$ in the graph $\Gamma$
%are joined by an edge, that is, $m_{ij}< \infty$.

\end{lemma}

\section{Conjugacy in dihedral Artin groups}

We recall that $G(m)$ denotes the dihedral Artin group
\[ \langle  x_1,x_2 \mid {}_m(x_1,x_2) = {}_m(x_2,x_1) \rangle. \]
This section is devoted to the proof of the case $n=2$ of
Theorem~\ref{thm:main}; that is, where $G=G(m)$.

\begin{proposition}
\label{dihedralconj}
The conjugacy of two elements represented as words of length at most $\ell$
in the dihedral Artin group $G(m)$
group of large type can be decided constructively in time $O(\ell)$.
\end{proposition}
\label{sec:dihedral}

\begin{proof}
For $m$ even,
$$G(m) \cong \langle y_1,y_2 \mid y_1y_2^{m/2} = y_2^{m/2}y_1\rangle.$$
with $y_1=x_1,y_2=x_1x_2$. So, the quotient of $G(m)$ by
its  central subgroup $\langle z:= y_2^{m/2} \rangle$
is a free product $C_{m/2} \star C_\infty$.  Since
that free product has trivial centre, we see that $Z:= \langle z\rangle$ is the centre
of $G(m)$.

For $m$ odd,
$$G(m) \cong \langle y_1,y_2 \mid y_1^2=y_2^m \rangle$$
with $y_1={}_m(x_1,x_2),y_2=x_1x_2$.
So, now the quotient of $G(m)$ by the central subgroup
$\langle z:= y_2^m \rangle$ is a free product $C_2 \star C_m$, and again 
$Z:= \langle z\rangle$ is the centre of $G(m)$.

In either case, and also when $m$ is infinite (and so $Z=1$),
$G(m)/Z$ is a free product of two cyclic groups and so,
for $g \not\in Z$, $C_{G(m)/Z}(gZ)$ is cyclic. So the inverse image of
$C_{G(m)/Z}(gZ)$ in $G(m)$ is abelian; this means in particular that $g$ cannot
be conjugate within $G(m)$ to any other element of $gZ$.

Conjugacy testing in free products is straightforward, by \cite[Theorem 4.2]{MKS}:
in the case of a free product of two cyclic groups, two elements are conjugate if
and only if the normal form of one is (essentially) a cyclic conjugate of that
of the other.  Suppose $g,h \in G(m)$ are given. If $g$ and $h$ are conjugate,
then so are the elements $gZ$ and $hZ$ of $G/Z$, and in that case,
for any element $xZ$ with $gZ=hZ^{xZ}$, by the final statement in the
preceding paragraph, we must have $g = h^x$. Hence, to test for conjugacy of
$g$ and $h$, we first test for conjugacy of $gZ$ and $hZ$; then, having found
some $x$ with $gZ=hZ^{xZ}$, we check whether $h^x=g$. If $gZ$ and $hZ$
are non-conjugate, or otherwise if for $x$ as above
$h^x \neq g$, then $g$ and $h$ are proved non-conjugate.

The steps in this process are
\begin{enumerate}
\item conversion of input words written over the standard
generating set $\{x_1,x_2\}$ to words over $\{y_1,y_2\}$;
\item collection of the resulting
words into a free product normal form modulo $Z$, followed by a power of $z$;
\item checking whether the free product normal form of one of the
words is a cyclic conjugate of the other and the powers of $z$ are the same.
\item writing down a conjugator as a word over $\{y_1,y_2\}$, and then rewriting
it as a word over $\{x_1,x_2\}$.
\end{enumerate}
Each of these four steps can be done in linear time 
\end{proof}

\section{Conjugates of powers of a generator}
\label{sec:genpowers}

The conjugates of powers of a generator are easy to recognise in
dihedral Artin groups.

\begin{proposition}\label{dhgenconj}
Suppose that $h$ is cyclically reduced,
and conjugate to $g:= x_1^k$ in $G(m)$,
for some $k\neq 0$.
If $h\neq g$, then $m$ is odd and $h=x_2^k$.
\end{proposition}
\begin{proof}
Recall that if $m$ is odd, then the element $\Delta= {}_m(x_1,x_2)$
conjugates $x_1$ to $x_2$; if $m$ is even, then $\Delta$ is central.

Let $f^{-1}gf = h$. The proof is by induction on $|f|$.
The case $|f|=0$ is clear, so assume $|f| > 0$.
We may assume that $g = x_1^k$ with $k \ne 0$.
Now $f$ does not have $x_1^{\pm 1}$ as a left divisor,
or else we could find a shorter conjugator, so $gf$ and $f^{-1}g$ are geodesic
factorisations.  But $f^{-1}gf$ cannot be a geodesic factorisation, or else
$h$ would not be cyclically reduced.

Let $u$ be a geodesic word for $f$. Then  $\f{u} = x_2^{\pm 1}$, so
$u^{-1}x_1^ku$ is freely reduced. Let $p:=p(u)$, $n:=n(u)$,
as defined in Section~\ref{sec:artin}.  
We have $p,n<m$, since otherwise $f$ has $\Delta^{\pm 1}$ as a left divisor,
and hence also $x_1^{\pm 1}$.

Now if $p>0$, then $u^{-1}x_1^ku$ contains a subword
$w={}_p(x^{-1},y^{-1})\xi (z,t)_p$,
where $\{x,y\}=\{z,t\} = \{x_1,x_2\}$, and the subword $x_1^k$ representing $g$ is
within $\xi$. This is equal in $G$ to the word
$w'={}_{m-p}(y,x)\delta(\xi) (t^{-1},z^{-1})_{m-p}$,

If $p > m/2$ then $w'$ is shorter than $w$,
and so we can find a representative
of $h$ of the form $u'^{-1}\delta(x_1^k)u'$
that has $w'$ as a subword 
and for which $u'$ is shorter than $u$.
So $h = f'^{-1} g' f'$ with $|f'| < |f|$ and $g'=\delta(g)$
conjugate to $g$.
The result then follows by the inductive hypothesis.
Similarly the result follows when $n > m/2$.

Since $u^{-1}x_1^ku$ is non-geodesic, we have
$p(u^{-1}x_1^ku)+ n(u^{-1}x_1^ku) > m$.
This inequality cannot hold if both
$p<m/2$ and $n<m/2$.
Hence to complete the proof it remains to consider the
cases where $m$ is even and
either $p=l$ or $n=l$, where $l=m/2$.

Assume that $k>0$ (the case $k<0$ is similar). Then $p(u^{-1}x_1^ku)+
n(u^{-1}x_1^ku) >  m$ is only possible if $p=l$ and $u$ has a prefix
${}_{l}(x_2,x_1)$ or ${}_l(x_2^{-1},x_1^{-1})$.
In the former case, we have
\[
 {}_{l}(x_2,x_1)^{-1}\, x_1^k\, {}_{l}(x_2,x_1) =
 (x_1^{-1},x_2^{-1})_l\, x_1^{k-1}\, {}_{l+1}(x_1,x_2)\]
The right hand side is an over-critical word and so
can be reduced as described in Section~\ref{sec:artin}.
When $l$ is even we express the right hand side as
 ${}_l(x_1^{-1},x_2^{-1})\, x_1^{k-1}\, (x_2,x_1)_{l+1}$ and see 
that
\begin{eqnarray*}
 {}_l(x_1^{-1},x_2^{-1})\, x_1^{k-1}\, (x_2,x_1)_{l+1}&=_G&
 {}_l(x_2,x_1)\, x_1^{k-1}\, (x_1^{-1},x_2^{-1})_{l-1}\\&=&
 {}_l(x_2,x_1)\, x_1^{k-1}\, {}_{l-1}(x_2,x_1)^{-1}\\&=&
 {}_{l-1}(x_2,x_1)\, x_1^k\, {}_{l-1}(x_2,x_1)^{-1},\end{eqnarray*}
while if $l$ is odd we express the right hand side as
 ${}_l(x_2^{-1},x_1^{-1})\, x_1^{k-1}\, (x_1,x_2)_{l+1}$ and see
that
\begin{eqnarray*}
 {}_l(x_2^{-1},x_1^{-1})\, x_1^{k-1}\, (x_1,x_2)_{l+1}&=_G&
 {}_l(x_1,x_2)\, x_1^{k-1}\, (x_2^{-1},x_1^{-1})_{l-1}\\&=&
 {}_l(x_1,x_2)\, x_1^{k-1}\, {}_{l-1}(x_1,x_2)^{-1}\\&=&
 {}_{l-1}(x_1,x_2)\, x_1^k\, {}_{l-1}(x_1,x_2)^{-1}.\end{eqnarray*}
In either case,
we have replaced $f$ by a shorter element, and the result follows by
the inductive hypothesis.
A similar argument applies when $u$ has the prefix
${}_l(x_2^{-1},x_1^{-1})$.
\end{proof}

From now on, and for the remainder of this article,  suppose that
$G = \langle x_1,\ldots,x_n \rangle$ is an Artin group of extra-large type.

Note that when $n\geq 3$ two generators $x_i,x_j$ of $G$ can be
conjugate even when $m_{ij}$ is even. This happens if there is a
sequence of generators
$x_i = x_{i_1},x_{i_2},\ldots,x_{i_k}=x_j$
with each $m_{i_t,i_{t+1}}$ odd
(and so $\Delta_{i_ti_{t+1}}:=(x_{i_t},x_{i_{t+1}})_{m_{i_ti_{t+1}}}$ conjugates $x_{i_t}$ to $x_{i_{t+1}}$). In that case, we say that $x_i^k$ and
$x_j^k$ are {\em generator conjugate}. 

The following proposition generalises Proposition~\ref{dhgenconj} to show
that the conjugates of powers of a generator are straightforward to recognise.
In fact the proof of this result is valid assuming only that $G$ has large
type.
We recall the definition of $\LD_{ij}(g)$ and $\RD_{ij}(g)$ from
Section~\ref{sec:artin} as the longest left and right divisors of $g$ in
$G(i,j)$ respectively, where $G(i,j) := \langle x_i,x_j \rangle < G$.

\begin{proposition}\label{genconj}
Suppose that $h$ is cyclically reduced,
and conjugate to $g:= x_i^k$ in $G$,
for some $k\neq 0$.
If $h\neq g$, then $h=x_j^k$  for some $j$,
and $g$ and $h$ are generator conjugate.
\end{proposition}
\begin{proof}
Choose an element $f$ such that $f^{-1}gf = h$. The proof is by
induction on $|f|$.
The case $|f|=0$ is clear, so assume $|f| > 0$.
We may assume that $g = x_j^k$ with $k \ne 0$.

The result is clear if $f$ is a power of a generator so suppose not, and
let $u$ be a geodesic representative of $f$.

Suppose that neither of the first two generators occurring in any
geodesic representative of $f$ is $x_i$.
By Proposition~\ref{ngreduceprop} (ii),
the word $u^{-1}gu$ must admit a rightward reducing sequence that starts
either in $u^{-1}$ or in $g$. But our
assumption ensures that $u^{-1}gu$ can contain no critical subword that
intersects $g$. So the sequence must start wholly within $u^{-1}$, 
and in that case some word $u'^{-1}gu$, for which $u'$ is a geodesic representative of $f$, must contain a critical subword intersecting
$g$ and $u'^{-1}$. Again our assumptions forbid that.

So now (if necessary replacing $u$ by another geodesic representative) we may
assume that one of the first two generators of $u$ is $x_i$.
Let $x_j$ be the other one, and let
$f=de$ with $d := \LD_{ij}(f)$.
If $g' := d^{-1}gd$ is a power of a generator, then $g$ and $g'$ are generator
conjugate by Proposition~\ref{dhgenconj}. We then have $h = e^{-1}g'e$
and the result follows by the inductive hypothesis, since $|e| < |f|$.

So $d^{-1}gd$ is not a power of a generator and hence, by
Proposition~\ref{dhgenconj}, it cannot be cyclically reduced, and so it has a
geodesic representative (of the form $ava^{-1}$) with at least three syllables.
Any other geodesic representative can be transformed to this by a sequence of
$\tau$-moves, so must contain a critical subword and hence also at least three
syllables. 
But then, by Lemma~\ref{geodfaclem1},
$f^{-1}gf$ is a geodesic
factorisation, contradicting the fact that $h$ is cyclically reduced.
\end{proof}

\section{Conjugacy of words involving two generators}
\label{sec:twogen}

\begin{proposition}\label{2genconj}
Let $g,h$ be cyclically reduced and conjugate in $G$ with $g \in G(i,j)$ for
some $i,j$, where $g$ is not a power of a generator. Then $h \in G(i,j)$ and
$g,h$ are conjugate in $G(i,j)$.
\end{proposition}
\begin{proof}
Let $h = f^{-1}gf$ where $f \not\in G(i,j)$.
%The proof is by  induction on $|f|$.
%The case $|f|=0$ is clear, so assume $|f| > 0$.

Let $f=de$ with $d := \LD_{ij}(f)$. So $e \ne 1$.
Then, since $g$ is not a power of a
generator, by Lemma~\ref{dhgenconj} $g' := d^{-1}gd$ cannot be
a power of a generator, so it has at least two syllables.
Now by Lemmas~\ref{geodfaclem1} and  \ref{geodfaclem2}, applied with
$g_1=e^{-1},g_2=g',g_3=e$, either
\begin{mylist}
\item[(i)] $e^{-1} g' e$ is a geodesic factorisation; or
\item[(ii)] 
$g' =_G a^sb^t$ with (without loss of generality) $a=x_i^{\pm 1}, b=x_j^{\pm
1}$ and $s,t>0$. Furthermore, for some letter $c$ with name $x_k\neq x_i,x_j$,
we have $e^{-1}a^s =_G e_1c^q$ and $b^te =_G c^{-q}e_2$, where $e_1c^q$,
$c^{-q}e_2$ and $e_1e_2$ are geodesic factorisations.
\end{mylist}
In Case (i), $h$ is not cyclically reduced, contrary to assumption. 

In Case (ii),
let $g''$ be the group element $b^te$. Then $\LD_{jk}(g'')$ has both $b^t$
and $c^{-q}$ as left divisors, and so by Lemma~\ref{dihedral_geodesics}
for any geodesic representative $w$ of $\LD_{jk}(g'')$ we have $p(w)+n(w)=m\geq 4$,
and so $w$ has at least four syllables. Then any representative of $\LD_{jk}(e)$
has at least three syllables, the first of which must be a power of
$x_k$ (recall that by construction no power of $x_i$ or $x_j$ can be a left
divisor of $e$).
Similarly, considering the group element $a^{-s}e=_Gc^{-q}e_1^{-1}$, 
we deduce that $\LD_{ik}(e)$ has at least three syllables, the first of those also
a power of $x_k$. Now for some $r$, $x_k^re$ has distinct geodesic
representatives with distinct first letters, one of which starts
with a word over $x_i,x_k$ and the other with a word over $x_j,x_k$.
This contradicts Proposition~\ref{twoposs}.
\end{proof}

\section{Complexity}
\label{sec:complexity}

We need to analyse the complexity of our algorithm in the case $n \geq 3$.
Recall that $G$ is an $n$-generator Artin group of extra-large type,
and so has a geodesic normal form arising from its biautomatic structure.
We prove
\begin{proposition}
\label{prop:complexity}
When $n \geq 3$ the algorithm described in the proof of Theorem~\ref{thm:main}
decides conjugacy between two words of length at most $\ell$ in an $n$-generator
Artin group in time $O(n\ell^3)$.
\end{proposition}
\begin{proof}
Let $u,v$ be words of length
at most $\ell$.
We analyse each step of the algorithm.
\begin{enumerate}
\item
\label{uvnf}
Reducing $u$ and $v$ to normal form takes time $O(\ell^2)$.
\item
\label{ccuvnf}
We reduce all cyclic conjugates of $u$ and $v$ to normal form.  If $u =
a_1a_2\cdots a_t$ with $t \le \ell$ and, for $1 \leq k \leq t$, $u_k$ is the
cyclic conjugate $a_k a_{k+1}\cdots a_t a_1 \cdots a_{k-1}$ of $u$ and we
have computed the normal form $u_k'$ of $u_k$, then we can compute $u_{k+1}'$
as the normal form of $a_k^{-1} u_k' a_k$ which, by biautomaticity, can be
done in time $O(\ell)$. So the total time for this step is $O(\ell^2)$.  
\item
\label{ccccuvnf}
We check that the cyclic conjugates of $u$ and $v$ are all cyclically
reduced as group elements, by checking $|au_k'a^{-1}| \geq |u_k'|$ and
$|av_k'a^{-1}| \geq |v_k'|$ for all $k$ and all $a \in X \cup X^{-1}$
involved in $u$ and $v$.  This takes time $O(n\ell^2)$.
\item
\label{ccshort}
If any of the normal form words for $u_k'$ (or $v_k'$) or any of
their conjugates under $a \in X \cup X^{-1}$ in Steps~\ref{ccuvnf}
or~\ref{ccccuvnf} is shorter than $u$ (or $v$), then replace $u$ (or $v$) by
this shorter word and restart. Since this can happen at most $\ell$ times,
the total time so far is $O(n\ell^3)$.

At this stage $u$ and $v$ are both specially cyclically reduced,
and $|u_k'| = |u|$, $|v_k'|=|v|$ for all $k$.

\item
\label{1genalg}
If one of $u$ is a power $x_i^k$, then we can check in time $O(\ell)$ whether
$v$ is also a $k$-power $x_j^k$ of a generator.
In that case $u$ and $v$ are conjugate if $i$ and $j$ are in the same component
of the graph formed from the Coxeter graph by deleting even-labelled edges;
those components can be computed in time $O(n^2)$ in pre-processing time
before the algorithm takes input.
\item
\label{2genalg}
If $u,v \in G(i,j)$ for some $i,j$, then we are in the 2-generator case, and
we proved in Section~\ref{sec:dihedral} that we can test $u$ and $v$ for conjugacy
in time $O(\ell)$.

Otherwise, by Theorem~\ref{thm:AppelSchupp},
%and Propositions~\ref{genconj},~\ref{2genconj}, ????
if $u$ and $v$ are conjugate, then
they must both involve the same set of at least three generators. If not,
then we return false.

At this stage we can apply Theorem~\ref{thm:AppelSchupp} and
Proposition~\ref{prop:boundlen}, which is stated and proved below.
In fact Proposition~\ref{prop:boundlen} implies that if $u$ and $v$ are
conjugate then $|u|=|v|$, so we return false if not.

\item
\label{ccvfsa}
Initialise an automaton for string searching for one of the normal
forms $v_k'$ of the cyclic conjugates $v_k$ of $v$.
By
\cite{AhoCorasick,BoyerMoore},
this takes time $O(\ell^2)$.

\item
\label{powergenconj}
For each letter $a \in X \cup X^{-1}$ that is involved in $u$, each normal
form $u_k'$ of the cyclic conjugate $u_k$ of $u$, and each $j$ with
$1 \leq j \leq |u|$, compute the normal form of $a^j u_k' a^{-j}$, and
use the automaton to check whether this equals any of the normal forms
$v_k'$ of the cyclic conjugates of $v$. By biautomaticity, computing
each of these normal forms takes time $O(\ell)$ as does the search for
equality with $v_k$, so the total time for this step is $O(n\ell^3)$.

(Note that we have already computed these normal forms when $j=1$ in
Step \ref{ccccuvnf}, so we could avoid that repetition.) 
\end{enumerate}
\end{proof}

\begin{proposition}
\label{prop:boundlen}
Suppose that $g,h$ are cyclically reduced elements in $G$, and suppose that,
for some generator $a \in X \cup X^{-1}$ and some $N \in \N$, $a^Nga^{-N} = h$.
Then $|g|=|h|$, and either $g=h$ or $N \leq |g|$.
\end{proposition}
\begin{proof}
Let $g=a^rg_0a^s$ with $r,s \in \Z$ be a geodesic factorisation of $g$, where
$g_0$ has neither $a$ nor $a^{-1}$ as left or right divisor. Since $g$ is
cyclically reduced, $r$ and $s$ cannot have opposite signs, and we assume
without loss of generality that $r,s \geq 0$.
If $N \leq s$ then $N \leq |g|$, and $h = a^{r+N} g_0 a^{s-N}$ is a geodesic
factorisation of $h$, so $|g|=|h|$ and we are done. So assume that $N-s > 0$. 

Then $h=a^{N+r}g_0a^{-N+s} = a^{N-s}a^{r+s}g_0a^{-(N-s)}$.
This last product cannot be geodesic, since $h$ is cyclically reduced;
hence neither is the product $ag_0a^{-1}$, by Corollary~\ref{twoposscor}.
Since, by assumption, $g_0$ has neither $a$ nor $a^{-1}$ as left or right
divisor, the products $ag_0$ and $g_0a^{-1}$ are geodesic, and so
$|g_1| = |g_0|$, where $g_1=ag_0a^{-1}$.

We claim that $g_1$ cannot have $a$ or $a^{-1}$ as left or right divisor.
If $g_1$ has $a$ as left divisor, then we get $|a^{-1}g_1| < |g_1|$, but
$a^{-1}g_1 = g_0a^{-1}$, so this contradicts $|g_0a^{-1}| = |g_0|+1=|g_1|+1$.
If $g_1$ has $a^{-1}$ as left divisor, then $g_1a$ has a geodesic representative
beginning with $a^{-1}$ and another one (arising from $g_1a=ag_0$) beginning
with $a$, which is impossible by 
Proposition~\ref{twoposs}.
The proof for right divisors is similar.

Continuing in this way, we find a
sequence $g_0,g_1,\ldots,g_{N-s}$ of elements, all of length $|g_0|$,
such that $ag_ia^{-1} = g_{i+1}$ for $0 \leq i < N-s$, where none of
the $g_i$ have $a$ or $a^{-1}$ as left or right divisor.
Since $h = a^{r+s}g_{N-s}$, we have $|g|=|h|$.

In fact, since $ag_i$ and $g_ia^{-1}$ are geodesic factorisations but
$ag_ia^{-1}$ is not, each $ag_i$ must admit a rightward critical factorisation
$(aw^{(i)}_1)w^{(i)}_2 \cdots w^{(i)}_k$,
as defined in Section~\ref{sec:artin} (with 
$u_1=aw^{(i)}_1,u_2=w^{(i)}_2,\ldots$),
which the associated 
rightward sequence of $\tau$-moves transforms to 
$g_{i+1}a=w^{(i+1)}_1w^{(i+1)}_2 \cdots w^{(i+1)}_ka$. 
For each $j$, all of the words $w_j^{(i)}$ involve the same pair of generators.
Setting $a_0:= a$, $a_1:=\l{\tau(aw_1^{0})}$ and, for each $j$,
$a_j:=\l{\tau(a_{j-1}w_j^{(0)})}$, we see that also, for each $i< N-s$,
$a_1=\l{\tau(aw_1^{(i)})}$ and, for each $j$, $a_j=\l{\tau(a_{j-1}w_j^{(i)})}$.
For each $j$ with $1\leq j\leq k$, we apply Lemma~\ref{bound} below to the
sequence $w_j^{(0)},w_j^{(1)},\ldots, w_j^{(N-s)}$, with $a:=a_{j-1}$, $b:=a_j$.

If for some $j$, $N-s \leq |w_j^{(i)}|$, then certainly $N-s \leq |g_i|=|g_0|$,
and so $N \leq |g_0|+r+s=|g|$.
Otherwise Lemma~\ref{bound} ensures that for each $i,j$ $w_j^{(i)}=w_j^{(i+1)}$,
and so $g_i=g_{i+1}=g_0$; in that case $g=aga^{-1}=h$.
\end{proof}

\begin{lemma}
\label{bound}
Let $w=w^{(0)},w^{(1)},\ldots,w^{(N)}$ be a sequence of words in the dihedral
Artin group $G(m)$, and let $a,b \in \{x_1,x_2\}^{\pm 1}$ such that for each
$i$ with $0 \le i < N$, $aw^{(i)}$ is critical
and $\tau(aw^{(i)})=w^{(i+1)}b$.
Then either $w^{(i)}=w^{(i+1)}$ for all $i$ or $N \leq |w|/(m-1)$.
\end{lemma}
\begin{proof}
Since $aw^{(i)}$ is critical and hence has length at least $m$,
the result is immediate if
$N \leq 1$. So we assume that $N>1$.

We consider the various possibilities for $w=w^{(0)}$.
Without loss of generality we assume that $a$ is a generator, rather than the
inverse of one.
From the definition of critical words in \cite[Section 2]{HoltRees},
we see that there are two possibilities for $aw$ and $\tau(aw)$ where,
in both cases, $\{a,c\}=\{b,d\}=\{x_1,x_2\}$ is the generating set. 

{\bf Case 1.} $aw ={}_p(a,c)\,\xi\,(b^{-1},d^{-1})_n$ with $p,n>0$ and $p+n=m$;
$\tau(aw) = {}_n(c^{-1},a^{-1})\,\delta(\xi)\,(d,b)_p$.

{\bf Case 2.}  $aw = {}_m(a,c)\,\xi$ with $\xi$ a positive word and
$\l{\xi} =d$; $\tau(aw)= \delta(\xi)\,(d,b)_m$.

We restrict the possibilities for $w$ by using the fact that $\p{\tau(aw)}$ is
the maximal suffix of the critical word $aw^{(1)}$.

Suppose first that we are in Case 1. The fact that
$aw^{(1)}$ must start with a positive alternating subword
of length $p$ forces $p=1$, and so $n=m-1$.    
If $\xi$ is non-empty then, since $aw^{(1)}$ must end with a negative
alternating subword of length $m-1$ that ends in $d^{-1}$, so must
$\delta(\xi)$; and so $\xi$ has a suffix equal to
$\delta((b^{-1},d^{-1})_{m-1})$.  We factorise $\xi$ as
$\eta \prod_{j=0}^{J-1}(\delta^{J-j}((b^{-1},d^{-1})_{m-1}))$,
for some word $\eta$, where $J> 0$ is as large as possible.
If $\xi$ is empty then we define $\eta$ to be the empty word and $J=0$.

Then
\begin{eqnarray*}
aw &=& a\eta \prod_{j=0}^J(\delta^{J-j}((b^{-1},d^{-1})_{m-1}),\\
aw^{(1)} &=& a\,{}_{m-1}(c^{-1},a^{-1})\,\delta(\eta)
    \prod_{j=0}^{J-1}(\delta^{J-1-j}((b^{-1},d^{-1})_{m-1}).
\end{eqnarray*}
If $\eta$ is empty then, since $aw$ is freely reduced, we see that
its second letter must be $c^{-1}$ rather than $a^{-1}$, and so
$\delta^{J}(b^{-1},d^{-1})_{m-1})={}_{m-1}(c^{-1},a^{-1})$,
and we have $aw=aw^{(1)}$ and then $w=w^{(i)}$ for all $i$.

Otherwise, we see that
\begin{eqnarray*}
aw^{(J)} &=& a \prod_{j=0}^{J-1} \delta^j({}_{m-1}(c^{-1},a^{-1}))\delta^J(\eta)
((b^{-1},d^{-1})_{m-1})),\\
\tau(aw^{(J)}) &=& 
\prod_{j=0}^J \delta^j({}_{m-1}(c^{-1},a^{-1}))
\delta^{J+1}(\eta)b,\\
aw^{(J+1)} &=& a\prod_{j=0}^J \delta^j({}_{m-1}(c^{-1},a^{-1}))
\delta^{J+1}(\eta).
\end{eqnarray*}
Now, since we chose $J$ to be maximal, $\delta^{J+1}(\eta)$ does not have
an alternating negative subword of length $m-1$ as a suffix.
So if $aw^{(J+1)}$ had such a suffix, then it would overlap the subword
$\prod_{j=0}^J \delta^j({}_{m-1}(c^{-1},a^{-1}))$. But then $aw^{(J+1)}$ would
contain an alternating negative subword of length at least $m$, and would be
be non-geodesic.
Hence, in this case we see that $aw^{(J+1)}$ is not critical,
and the sequence ends with $aw^{(J+1)}$. So $N \leq J+1 \leq |w^{(0)}|/(m-1)$.

Suppose now that we in Case 2.
If $\xi$ is empty then $\tau(aw)={}_m(c,a)$, and $w^{(i)}={}_{m-1}(c,a)$
for all $i$. So we may assume that $\xi$ is non-empty.

We write $\xi=\prod_{j=0}^{J-1} \delta^{j+1}({}_{m-1}(c,a))\,\eta$,
with $J\geq 0$ as large as possible.  Then 
\begin{eqnarray*}
aw &=& {}_m(a,c)\prod_{j=0}^{J-1} \delta^{j+1}({}_{m-1}(c,a))\,\eta,\\
aw^{(1)} &=& {}_m(a,c)\prod_{j=0}^{J-2} \delta^{j+1}({}_{m-1}(c,a))\,
\delta(\eta)\, (b,d)_{m-1}
\end{eqnarray*}

If $\eta$ is empty then, since $\l{w} = \l{w^{(1)}} = d$, we see that
$aw=aw^{(1)}$ and so $w=w^{(i)}$ for all $i$.
 
So suppose that $\eta$ is non-empty. If $m$ is even, then $\delta(\eta)=\eta$,
so $\l{\delta(\eta)} = \l{\eta} = d = \f{(b,d)_{m-1}}$. If $m$ is odd, then
$\l{\delta(\eta)} \ne \l{\eta}$, so $\l{\delta(\eta)} = b = \f{(b,d)_{m-1}}$.
So  $\l{\delta(\eta)} = \f{(b,d)_{m-1}}$ in either case.
We have
\begin{eqnarray*}
aw^{(J)}) &=& {}_m(a,c)\delta^J(\eta)\prod_{j=1}^{J}\delta^{J-j}((b,d)_{m-1}),\\
\tau(aw^{(J)}) &=& \delta^{J+1}(\eta)
\prod_{j=1}^{J}\delta^{J+1-j}((b,d)_{m-1}) (d,b)_m,\\
aw^{(J+1)} &=& a\delta^{J+1}(\eta)\prod_{j=1}^{J+1}\delta^{J+1-j}((b,d)_{m-1}).
\end{eqnarray*} 
Now, since we chose $J$ to be maximal, $a\delta^{J+1}(\eta)$ does not have
an alternating postive subword of length $m$ as a prefix and, since
$\l{\delta(\eta)} = \f{(b,d)_{m-1}}$, neither does $aw^{(J+1)}$.
So $aw^{(J+1)}$ is not critical, the sequence ends with $aw^{(J+1)}$, and
again we have $N \leq J+1 \leq |w^{(0)}|/(m-1)$.
\end{proof}

%\affiliationone{

Derek F. Holt\\
Mathematics Institute, University of Warwick,\\
Coventry CV4 7AL, UK.\\
dfh@maths.warwick.ac.uk

Sarah Rees,\\
School of Mathematics and Statistics, University of Newcastle,\\
Newcastle NE1 7RU, UK.\\
Sarah.Rees@ncl.ac.uk
\end{document}